\documentclass[a4paper,10pt]{amsart}

\usepackage[english]{babel}

\usepackage[letterpaper,top=2cm,bottom=2cm,left=3cm,right=3cm,marginparwidth=1.75cm]{geometry}

\usepackage{amsmath}
\usepackage{graphicx}
\usepackage[colorlinks=true, allcolors=blue]{hyperref}
\usepackage{comment}
\usepackage{amsmath,amsthm,amssymb,amsfonts,amscd, amsbsy,mathtools}

\usepackage{subcaption}

\newtheorem{example}{Example}[section]
\newtheorem{remark}{Remark}[section]

\newtheorem{theorem}{Theorem}[section]
\newtheorem{proposition}{Proposition}[section]
\newtheorem{corollary}{Corollary}[section]
\newtheorem{definition}{Definition}[section]

\newtheorem{example*}{Example}
\newtheorem{remark*}{Remark}
\newtheorem{lemma*}{Lemma}
\newtheorem*{theorem*}{Theorem}
\newtheorem*{proposition*}{Proposition}
\newtheorem*{corollary*}{Corollary}
\newtheorem*{definition*}{Definition}

\title{Rigidity of Translation Surfaces in the Three-dimensional Sphere $\mathbb{S}^3$ }
\author{Ferreira, T. A. and dos Santos, J. P.}

\address{Jo\~ao Paulo dos Santos - Departamento de Matem\'atica, 
Universidade de Bras\'ilia, 70910-900, Bras\'ilia-DF, Brazil}
\email{joaopsantos@unb.br}

\address{Tarcios Andrey Ferreira - Departamento de Matemática, Universidade de Brasília, 70910-900, Brasília-DF, Brazil }
\email{t.a.ferreira@mat.unb.br}

\date{\today}

\begin{document}



\begin{abstract}
A translation surface in the three-dimensional sphere $\mathbb{S}^3$ is a surface generated by the quaternionic product of two curves, called generating curves. In this paper, we present rigidity results for such surfaces. We introduce an associated frame for curves in $\mathbb{S}^3$,  and by means of it, we describe the local intrinsic and extrinsic geometry of translation surfaces in $\mathbb{S}^3$. The rigidity results, concerning minimal and constant mean curvature surfaces, are given in terms of the curvature and torsion of the generating curves and their proofs rely on the associated frame of such curves. Finally, we present a correspondence between translation surfaces in $\mathbb{S}^3$  and translation surfaces in $\mathbb{R}^3$. We show that these surfaces are locally isometric, and we present a relation between their mean curvatures. 
\end{abstract}

\maketitle

\section{Introduction} 

Translation surfaces in $\mathbb{R}^3$ are defined as the sum of two curves $\alpha: I \subset \mathbb{R} \to \mathbb{R}^3$ and $\beta: J \subset \mathbb{R} \to \mathbb{R}^3$. More precisely, they are a special case of a broader class of surfaces known as Darboux surfaces. Following \cite{LopezHasanis}, the origin of these surfaces dates back to \cite{Darboux}, where they are described as the motion of a curve under a one-parameter family of rigid motions in $\mathbb{R}^3$. A general parameterization of such a surface is given by $\Phi(s, t) = A(t) \cdot \alpha(s) + \beta(t)$, where $A(t)$ is an orthogonal matrix. In the particular case where $A(t)$ is the identity, the surface $S \subset \mathbb{R}^3$ can be locally written as the sum of two curves, $\Phi(s, t) = \alpha(s) + \beta(t)$, and is called a translation surface. The curves $\alpha$ and $\beta$ are referred to as the generating curves of $S$, and the terminology reflects the fact that the surface $S$ is obtained by translating one curve along the other.

In general, let $G$ be a Lie group with group operation denoted by ($\cdot$). A translation surface $S \subset G$ is a surface that can be locally written as the product $\Psi(s, t) = \alpha(s) \cdot \beta(t)$ of two curves $\alpha: I \subset \mathbb{R} \to G$ and $\beta: J \subset \mathbb{R} \to G$. The curves $\alpha$ and $\beta$ are referred to as the generating curves of $S$. This work is inspired by previous studies on minimal translation surfaces, such as \cite{LopezHasanis,Munteanu,Lopez,LopezMunteanu,LopezPerdomo,MoruzMunteanu,Yoon}. These earlier works primarily focus on the Thurston 3-dimensional geometries, many of which are also Lie groups. The aim of the present paper is to investigate constant mean curvature (CMC) and minimal translation surfaces in the 3-dimensional sphere $\mathbb{S}^3$.

It is well known that the unit 3-sphere $\mathbb{S}^3 \subset \mathbb{R}^4$ admits a Lie group structure equipped with a bi-invariant metric, when viewed through its quaternionic structure. This structure plays a crucial role in the theory of flat surfaces, from the classical Bianchi–Spivak construction (see \cite{Galvez} and \cite{Spivack4}) to the more sophisticated approach developed in \cite{Kitagawa}. It continues to be relevant today, as evidenced by recent works such as \cite{AledoGalvezMira,GalvezMira,ManfioSantos}.

Our contributions in this work are presented as rigidity results regarding the mean curvature. To achieve these results, in section \ref{section_preliminary} we establish the local geometry of generic translation surfaces in $\mathbb{S}^3$ by means of their generating curves (Theorem \ref{theo_prop_surface_general}). A critical element for understanding such a local geometry, and for subsequent results, is the introduction of a suitable frame field, which has its own interest (Definition \ref{def:frames}). From such a frame, geometric objects like the Gaussian curvature and the mean curvature can be fully described. It plays a fundamental role in proving our main results. 

In section \ref{SectionFlat}, we present some results revisiting flat surfaces. According to \cite{XuHan}, if $G$ is an $n$-dimensional Lie group ($n \geq 3$) equipped with a bi-invariant metric, and $S$ is a translation surface in $G$, that is, locally parametrized as the group product of two curves, with constant Gaussian curvature, then $S$ must be flat. Therefore, it suffices to consider the flat case. We apply this result, together with Gauss equation, to show that there exists no totally umbilic surface in $\mathbb{S}^3$ that can be written as a translation surface (Theorem \ref{TeoTotalGeo}). Furthermore, we apply the properties of the frame mentioned above to provide a sufficient condition for the product $\alpha \cdot \beta$ to to be flat. Indeed, when $\alpha$ and $\beta$ are curves parametrized by the arc length, we will see that the vector fields $T_\alpha = \overline{\alpha} \cdot t_\alpha$ and $T_\beta = \beta \cdot \overline{t}_\beta$ are well-defined and belong to the corresponding frame to $\alpha$ and $\beta$, playing a crucial role in this work. In this section, it is shown that if the angle between $T_\alpha$ and $T_\beta$ is constant then $\alpha \cdot \beta$ is flat. In this case, we establish a rigidity result for $\alpha$ and $\beta$ in terms of great circles and general helices in $\mathbb{S}^3$ (see \cite{Barros} for a reference to such curves). As a consequence, we provide a nice result for curves in $\mathbb{S}^3$.

Section~\ref{SectionCMC} is dedicated to the rigidity of the CMC  Clifford tori, including the minimal Clifford tori, through conditions imposed on the generating curves of a translation surface $\alpha \cdot \beta$. In particular, Theorem \ref{TeoCMCbothcurvatzero} establishes that the only translation surfaces generated by great circles are the Clifford tori. Furthermore, we present a correlation between $T_{\alpha}$ and $T_{\beta}$, and the precise value of the mean curvature in the case of translation surfaces generated by two great circles. More precisely, such mean curvature is entirely determined by the value of $\langle T_\alpha, T_\beta \rangle$, which is constant in this case.  The second result of this section (Theorem \ref{TeoCMCcurvatzero}) provides a rigidity result of a CMC translation surface as a CMC Clifford torus, taking into account the constancy of $\langle T_\alpha, T_\beta \rangle$ or the vanishing of the curvature of the generating curves.

In section \ref{SectionR3}, we present a correspondence between translating surfaces in the Euclidean 3-space $\mathbb{R}^3$ and in the 3-sphere $\mathbb{S}^3$. We show that these surfaces are locally isometric, and we present a relation between their mean curvatures. Namely, we have    
\begin{theorem*}[Theorem \ref{theoS3R3}]
    Let  $M \subset \mathbb{S}^3$ be a translation surface generated by curves  $\alpha$ and $\beta$ with curvatures $ \kappa_\alpha$, $ \kappa_\beta$ and, when $\kappa_\alpha \not\equiv 0$ and/or $\kappa_\beta \not\equiv 0$,  torsions $\tau_\alpha$ and  $ \tau_\beta$. Then this surface is locally isometric to a translation surface $\tilde{M} \subset \mathbb{R}^3$ generated by curves $\tilde{\alpha}$ and $\tilde{\beta}$ with curvatures $\tilde{\kappa}_\alpha = \kappa_\alpha  $, $\tilde{\kappa}_\beta = \kappa_\beta  $ and torsions $\tilde{\tau}_\alpha = (\tau_\alpha-1)$, $\tilde{\tau}_\beta = ( \tau_\beta+ 1 )$. The reciprocal identification is also true. Moreover, the mean curvatures $\tilde{H}$ and $H$ satisfy 
    $$ \tilde{H} = H + \frac{\langle T_\alpha , T_\beta \rangle}{\sqrt{1 - \langle T_\alpha , T_\beta \rangle^2 }} . $$
\end{theorem*}
\noindent Such a Theorem provides some interesting applications when compared to the results of \cite{LopezHasanis,LopezMunteanu}.

We finish this work focused on non-existence results for a class of translation minimal surfaces. Theorems \ref{theokconsttau-1} and \ref{Theo_min_2helix} in section \ref{section_minimal} establish the non-existence of minimal surfaces when non-vanishing curvatures and torsions of the generating curves are constant, providing the rigidity of the minimal Clifford tori through conditions on the generating curves.

\section{Preliminary Concepts}\label{section_preliminary}

\subsection{\texorpdfstring{$\mathbb{S}^3$}{TEXT} as a Lie group with the quaternionic model} In this subsection, we present the quaternionic model for $\mathbb{S}^3$, which equips it with the structure of a Lie group endowed with a bi-invariant metric. We also introduce basic concepts and properties that will be useful throughout this work. For further details, we refer the interested reader to \cite{Galvez,Spivack4}.

We begin by identifying $\mathbb{R}^4$ with the nonzero quaternions $\mathbb{H}^* = \mathbb{H} \setminus \{ 0 \}$ in the standard way: $(x_1, x_2, x_3, x_4)$ is viewed as the quaternion $x_1 + i x_2 + j x_3 + k x_4$.  Hence, for $x = (x_1,x_2, x_3, x_4) $ and $y = (y_1,y_2,y_3,y_4) $, we have 
$$ x \cdot y  = \begin{bmatrix}
    x_1y_1 - x_2y_2 - x_3y_3 - x_4y_4 \\
    x_1y_2 + x_2y_1 + x_3y_4 - x_4y_3 \\
    x_1y_3 - x_2y_4 + x_3y_1 + x_4y_2 \\
    x_1y_4 + x_2y_3 - x_3y_2 + x_4y_1
\end{bmatrix} .$$
We also define the conjugate of  $ x \in \mathbb{R}^4$ as $\overline{x} = (x_1, - x_2, - x_3 ,- x_4)$.

Now, let $x,y,a \in \mathbb{H}^*$. The following summarizes the properties of this group and they follow  from the definition of quaternions and the usual metric of $\mathbb{R}^4$
\begin{equation}\label{eq_quater_properties}
 \begin{array}{llllll}
   1. & \overline{x \cdot y} = \overline{y} \cdot \overline{x}  & 3. & \langle x \cdot y , x \cdot y \rangle = \langle x,x \rangle \langle y,y \rangle. \\
    2. &  \langle x \cdot a , y \cdot a \rangle = \langle x , y  \rangle. & 4. & x^{-1} = \overline{x}/|x|^2 . \\
\end{array}   
\end{equation}

Therefore, since $\mathbb{S}^3 = \{ (x_1,x_2,x_3,x_4) \in \mathbb{R}^4 \ | \ x_1^2 + x_2^2 + x_3^2 + x_4^2   = 1 \} $, it follows from the previously listed properties that, for all $x,y \in \mathbb{S}^3$, we have  $\langle x \cdot y, x \cdot y \rangle = 1$, that is, the product is closed in $\mathbb{S}^3$. Since this product is differentiable, it endows $\mathbb{S}^3$ with the structure of a Lie group, whose identity element is $e_1=(1,0,0,0)$. We also point out that $\mathcal{S} = (\{ 0\} \times \mathbb{R}^3) \cap \mathbb{S}^3$ can be seen as the space of purely imaginary unit quaternions and this notation will be important as the set $\mathcal{S}$ appears recursively throughout this work. Finally, we will use the notation $x \perp y$, for $x,y \in \mathbb{S}^3$, to indicate that $\langle x, y \rangle = 0$.

By the property 4 in \eqref{eq_quater_properties} we conclude that $x^{-1}=\overline{x}$ whenever $x \in \mathbb{S}^3$. Therefore, if $x \perp y$, then $\langle x \cdot \overline{y}, e_1 \rangle = 0$ and $\langle \overline{x}, \overline{y} \rangle = 0$. Moreover, if  $x_1 = y_1 = 0 $, then  
$$x \cdot y = (0, x_3y_4 - x_4y_3, x_4y_2 - x_2y_4,  x_2y_3 - x_3y_2). $$
Now let $\tilde{x} = (x_2,x_3,x_4), \ \tilde{y} = (y_2,y_3,y_4) \in \mathbb{R}^3$, and set  $x = (0,\tilde{x}) , y = (0,\tilde{y}) \in \mathbb{R}^4$. Thus $$x \cdot y = (0 , \tilde{x} \times \tilde{y} ),$$
where $\times$ denotes the cross product in $\mathbb{R}^3$.

\subsection{Frenet-Serret equations and special frames for curves in \texorpdfstring{$\mathbb{S}^3$}{TEXT}} 

In what follows, let $\nabla$ be the standard Levi-Civita connection in $\mathbb{S}^3$. Let $\alpha : I \subset \mathbb{R} \rightarrow \mathbb{S}^3$ be a smooth curve parametrized by the arc length. Following \cite[Chapter 7, Part B]{Spivack4}, we denote the tangent vector of $\alpha.$ by $t_\alpha=\alpha'$. The curvature of $\alpha$ is defined as $\kappa_{\alpha}(s):= \left| \nabla_{\alpha'(s)} \alpha'(s) \right|$. At the points $s$ where $\kappa_\alpha(s) \neq 0$, we define $n_\alpha (s)$ as $n_\alpha(s) = \kappa_{\alpha}^{-1}(s) \nabla_{\alpha'(s)} t_{\alpha}(s)$. Finally, at the points where both $t_\alpha$ and $n_\alpha$ are well defined, we define the binormal vector field to $\alpha$ as the unit vector in $T_{\alpha}\mathbb{S}^3$ that is orthogonal to both $t_\alpha$ and $n_\alpha$, and such that the frame $\left\{ t_\alpha, n_\alpha, b_{\alpha} \right\}$ is positively oriented with respect to the orientation of $\mathbb{S}^3$. Throughout this paper, we will consider the orientation on $\mathbb{S}^3$ such that the unit normal field is given by $N(p)=p$. In this case, $b_\alpha \in  T_{\alpha} \mathbb{S}^3$ defined so that $\det (\alpha, t_\alpha, n_\alpha, b_\alpha) >0$. 

The well-known Frenet-Serret equations for smooth curves in $\mathbb{S}^3$, parametrized by the arc length are given by
$$ \left\{ \begin{array}{lcl}
     \nabla_{t_\alpha}  t_\alpha & = &  \kappa_\alpha n_\alpha ,\\
     \nabla_{t_\alpha}  n_\alpha & = &  - \kappa_\alpha t_\alpha + \tau_\alpha b_\alpha ,\\
     \nabla_{t_\alpha}  b_\alpha & = &  -\tau_\alpha n_\alpha , \\
\end{array}\right. $$
where $\kappa_\alpha$ and $\tau_\alpha$ are the curvature and torsion of $\alpha$, respectively. Thus, from the definition of $\nabla$, we derive the following equations
\begin{equation}\label{eq_curv_alpha}
\begin{cases}
\alpha' = t_\alpha,  \\
\alpha'' = \kappa_{\alpha} n_{\alpha}-\alpha ,  
\end{cases}  \ \ \ \ \ 
\begin{cases}
{t_\alpha}' = \kappa_{\alpha} n_{\alpha} -\alpha ,  \\
n_\alpha'= - \kappa_{\alpha} t_{\alpha} + \tau_\alpha b_\alpha,  \\
b_\alpha'= - \tau_{\alpha} n_{\alpha} . \\
\end{cases} 
\end{equation}
Since $t_\alpha$ is a unit vector field, we define the vector field $T_\alpha$ as the product $T_\alpha := \overline{\alpha} \cdot t_\alpha$. If $\kappa_\alpha \neq 0$, the Frenet frame $\left\{ t_\alpha, n_\alpha, b_\alpha \right\}$ is well defined, and we can extend this construction to define the vector fields $N_\alpha := \overline{\alpha} \cdot n_\alpha$ and $B_\alpha := \overline{\alpha} \cdot b_\alpha$. It follows from \eqref{eq_quater_properties} that $\{ T_\alpha, N_\alpha, B_\alpha \}$ provides an orthonormal frame. In the context of translation surfaces, it will be also useful to consider the frame $\{\hat{T}_\alpha, \hat{N}_\alpha, \hat{B}_\alpha \}$ defined by  $\hat{T}_\alpha = \alpha \cdot \overline{t_\alpha}$, $\hat{N}_\alpha = \alpha \cdot \overline{n_\alpha}$ and $\hat{B}_\alpha = \alpha \cdot \overline{b_\alpha}$. Let us formalize this construction with the following definition:
\begin{definition}
    Let $\alpha : I \subset \mathbb{R} \rightarrow \mathbb{S}^3$ be an arc length curve with curvature $\kappa_\alpha \neq 0$ everywhere. A quaternionic frame associated with $\alpha$ is defined as the orthonormal set $\left\{ T_\alpha, N_\alpha, B_\alpha \right\}$, where $T_\alpha = \overline{\alpha} \cdot t_\alpha$, $N_\alpha = \overline{\alpha} \cdot n_\alpha$ and $B_\alpha = \overline{\alpha} \cdot b_\alpha$. Similarly, we define the frame $\{\hat{T}_\alpha, \hat{N}_\alpha, \hat{B}_\alpha \}$ by $\hat{T}_\alpha = \alpha \cdot \overline{t_\alpha}$, $\hat{N}_\alpha = \alpha \cdot \overline{n_\alpha}$ and $\hat{B}_\alpha = \alpha \cdot \overline{b_\alpha}$. We call these the left and right frames, respectively.  \label{def:frames}
\end{definition}

The next proposition provides useful identifications for the frames  $\left\{ T_\alpha, N_\alpha, B_\alpha \right\}$ and $\{\hat{T}_\alpha, \hat{N}_\alpha, \hat{B}_\alpha \}$.

\begin{proposition} \label{prop:frames}
    Let $\alpha(s)$, be an arc length curve in $\mathbb{S}^3$ with $\kappa_\alpha \neq 0 $. Then we have 
    $$ \begin{array}{lll}
         T_\alpha = \overline{b_\alpha} \cdot n_\alpha, &   N_\alpha = \overline{t_\alpha} \cdot b_\alpha, &  B_\alpha = \overline{n_\alpha} \cdot t_\alpha , \\
         \hat{T}_\alpha = - b_\alpha \cdot \overline{n_\alpha}, &   \hat{N_\alpha} = -t_\alpha \cdot \overline{b_\alpha}, & \hat{B_\alpha} =  - n_\alpha \cdot \overline{t_\alpha}.
    \end{array}
    $$
\end{proposition}
\begin{proof}
    Since $x \perp y $ implies in $x\cdot \overline{y} \in \mathcal{S}$,  $\langle \overline{x}, \overline{y} \rangle = 0$ and $x \cdot \overline{y} = - y \cdot \overline{x} $. Thus $\left\{ \overline{n_\alpha} \cdot b_\alpha, \overline{b_\alpha} \cdot t_\alpha,  \overline{t_\alpha} \cdot n_\alpha \right\} \subset \mathcal{S}$ is an orthonormal frame and  we have the following   
    $$\begin{array}{ccc}
    \langle \overline{\alpha} \cdot t_\alpha , \overline{b_\alpha} \cdot n_\alpha \rangle  = \pm 1, &  \langle \overline{\alpha} \cdot n_\alpha , \overline{t_\alpha} \cdot b_\alpha \rangle = \pm 1, &  \langle \overline{\alpha} \cdot b_\alpha , \overline{n_\alpha} \cdot t_\alpha \rangle =  \pm 1. 
    \end{array} $$
    All the other possible inner products vanish. 
        
    To determine the correct signs in the above products, note that the relations must hold for every configuration of the frame $ \left \{ t_\alpha, n_\alpha, b_\alpha \right\}$. Up to a rigid motion, we may assume that at a given point $s_0$ we have $\alpha(s_0)=e_1$, $t_\alpha(s_0)=e_2$, $n_\alpha(s_0)=e_3$ and $b_\alpha(s_0)=e_4$. In this case, it follows that  $T_\alpha = e_2$ and $\overline{b}_{\alpha} \cdot n_{\alpha} =e_2$. The other cases are similar. The other cases are analogous. Proceeding in this manner, and applying the same reasoning to the other possible vectors formed by $\alpha, t_\alpha, n_\alpha, b_\alpha$ and their products, we obtain
    \begin{center}
    \begin{tabular}{|c|c|c|c|}
    \hline
      $\langle \cdot , \cdot \rangle$  & $\overline{b}_\alpha \cdot  n_\alpha$  &  $\overline{t}_\alpha \cdot b_\alpha$ & $\overline{n}_\alpha \cdot t_\alpha$ \\
      \hline
      $T_\alpha $ &  $1$ &  0 & 0 \\
      $N_\alpha$ & 0 &  $ 1$ & $ 0$ \\
      $B_\alpha$ & 0 &  $ 0 $ & $ 1$ \\
      \hline 
    \end{tabular}
    \end{center}
    The procedure for the frame $\{\hat{T}_\alpha, \hat{N}_\alpha, \hat{B}_\alpha \}$ is analogous.
\end{proof}

Using the equations \eqref{eq_curv_alpha} we derive the following Frenet-Serret type equations for the quaternionic frame of $\alpha$:
\begin{equation}\label{propertiesalpha}
\begin{cases}\begin{array}{rccl}
T_\alpha' = & \overline{t_\alpha} \cdot t_\alpha + \overline{\alpha} \cdot (\kappa_\alpha n_\alpha -\alpha) & = & \kappa_\alpha N_\alpha,  \\
N_\alpha' = & \overline{t_\alpha} \cdot n_\alpha + \overline{\alpha} \cdot (-\kappa_\alpha t_\alpha +\tau_\alpha b_\alpha) & = & - \kappa_\alpha T_\alpha + (\tau_\alpha - 1 )B_\alpha , \\
B_\alpha' = & \overline{t_\alpha} \cdot b_\alpha + \overline{\alpha} \cdot (- \tau_\alpha n_\alpha) & = &  - ( \tau_\alpha - 1 )N_\alpha .  
\end{array}
\end{cases} 
\end{equation}
Moreover, for the frame $\{\hat{T}_\alpha, \hat{N}_\alpha, \hat{B}_\alpha \}$, we have 
\begin{equation}\label{propertiesbeta}
\begin{cases}\begin{array}{rccl}
\hat{T}_\alpha' = & t_\alpha \cdot \overline{t_\alpha} + \alpha \cdot \overline{(\kappa_\alpha n_\alpha - \alpha) } & = & \kappa_\alpha \hat{N}_\alpha  , \\
\hat{N}_\alpha' = & t_\alpha \cdot \overline{n_\alpha} + \alpha \cdot \overline{(-\kappa_\alpha t_\alpha + \tau_\alpha b_\alpha)} & = & - \kappa_\alpha \hat{T}_\alpha + (\tau_\alpha + 1 ) \hat{B}_\alpha ,  \\
\hat{B}_\alpha' = & t_\alpha \cdot \overline{b_\alpha} + \alpha \cdot \overline{(- \tau_\alpha n_\alpha)} & = & - (\tau_\alpha + 1 ) \hat{N}_\alpha  .
\end{array}\end{cases}
\end{equation}

\subsection{Geometry of translation surfaces in \texorpdfstring{$\mathbb{S}^3$}{TEXT}} Let  $\alpha: I \subset \mathbb{R} \rightarrow \mathbb{S}^3$, $\alpha(s)$ and $\beta: J \subset \mathbb{R}  \rightarrow \mathbb{S}^3$, $\beta(t)$ be two arc length curves. Consider the map 
$$\begin{array}{rccl}
     X : &  I \times J & \rightarrow  & \mathbb{S}^3 \\
        & (s,t) & \mapsto & \alpha(s)\cdot \beta(t)
\end{array} $$
Since $\partial_s X(s,t) = \alpha'(s)\cdot \beta(t)$ and $\partial_t X(s,t) = \alpha(s)\cdot \beta'(t)$ are non-null vectors, the condition for $X$ to be a regular parametrization of a surface in $ \mathbb{S}^3$ is  $\langle \alpha'(s)\cdot \beta(t), \alpha(s)\cdot \beta'(t) \rangle \neq \pm 1 $.

Let $X : I \times J \rightarrow  \mathbb{S}^3$, $X(s,t)=\alpha(s)\cdot \beta(t)$ be a parametrization of a translation surface. From now on, we will always use the left frame for the curve $\alpha$ and the right frame for the curve $\beta$. To simplify the notation, the structure of the following computations will allow us to denote the right frame $\{\hat{T}_\beta, \hat{N}_\beta, \hat{B}_\beta \}$ of the curve $\beta$ as $\{T_\beta, N_\beta, B_\beta \}$ without risk of confusion. Moreover, the parameters $s$ and $t$ will be omitted throughout the calculations to make the presentation clearer and more pleasant for the reader. Also, from now on, we will always assume that $\alpha$ and $\beta$ are parametrized by the arc length.

Furthermore, throughout this work, the results are stated for $\alpha(s) \cdot \beta(t)$ but also hold for $\beta(t) \cdot \alpha(s)$, unless said otherwise. In particular, for the results involving $\tau_\alpha = 1 $,  the corresponding statements hold with the roles of $\alpha$ and $\beta$ interchanged, but with $\tau_\beta = -1$.

Now we present the following
\begin{theorem}\label{theo_prop_surface_general}
    Let $X : I \times J \rightarrow  \mathbb{S}^3$ ,  $X(s,t)=\alpha(s)\cdot \beta(t)$, be a parametrization of a translation surface. Then the regularity condition is given by
\begin{equation}
    \langle T_\alpha, T_\beta \rangle \neq \pm 1. \label{regularity}
\end{equation}
    The unit normal field at $X(s,t)$ in $\mathbb{S}^3$ is  
    \begin{equation}\label{GeneralNormal}
        N(s,t) =  \dfrac{\alpha'(s) \cdot \beta'(t) - \langle T_\alpha(s) , T_\beta(t) \rangle \ \alpha(s)\cdot\beta(t)}{\sqrt{1 - \langle T_\alpha(s) , T_\beta(t) \rangle^2}} .
    \end{equation}
    The mean curvature is given by 
    \begin{equation}\label{Meancurvature}
        H  =  \dfrac{\kappa_{\alpha} \langle B_{\alpha} , T_\beta \rangle - \kappa_{\beta} \langle T_\alpha, B_{\beta}  \rangle - 2 \langle T_\alpha , T_\beta \rangle [ \langle T_\alpha , T_\beta \rangle^2 - 1] }{2[ 1 - \langle T_\alpha , T_\beta \rangle^2 ]^{3/2}}.
    \end{equation}
    Similarly, the Gaussian curvature is expressed as
    \begin{equation}\label{Gaussiancurvature}
         K = \frac{\kappa_\alpha\kappa_\beta\langle B_\alpha , T_\beta \rangle \langle T_\alpha, B_\beta \rangle }{(1 - \langle T_\alpha , T_\beta \rangle^2)^2}.
    \end{equation}
\end{theorem}

\begin{proof} 
Initially, with equations \eqref{propertiesalpha} and \eqref{propertiesbeta}, we compute the coefficients of the first fundamental form 
\begin{equation}\label{firstfundamnetalform}
    \begin{array}{rllll}
    E & = \langle X_s, X_s \rangle & =  \langle \alpha'\cdot\beta , \alpha'\cdot\beta \rangle & = 1 , \\ 
    G & = \langle X_t, X_t \rangle & =  \langle \alpha\cdot\beta' , \alpha\cdot\beta' \rangle & = 1 , \\ 
    F & = \langle X_s, X_t \rangle & =  \langle \alpha'\cdot\beta , \alpha\cdot\beta' \rangle & = \langle T_\alpha , T_\beta \rangle .
    \end{array}
\end{equation}
Set $Y(s,t) = \alpha'(s) \cdot \beta'(t)$. Hence 
$$\begin{array}{ccccccc}
     \langle X_s(s,t),Y(s,t) \rangle & = & \langle \alpha'(s) \cdot \beta(t), \alpha'(s) \cdot \beta'(t) \rangle & = & \langle \beta(t) , \beta'(t) \rangle & = & 0 , \\
     \langle X_t(s,t),Y(s,t) \rangle & = & \langle \alpha(s) \cdot \beta'(t), \alpha'(s) \cdot \beta'(t) \rangle & = & \langle \alpha(s) , \alpha'(s) \rangle & = & 0 .
\end{array} $$
Thus,  $Y(s,t)$ is orthogonal to $X_s$ e $X_t$, for every $s \in I, t \in J$. Also, by a similar argument, $X(s,t)$ is also orthogonal to $X_s$ and $X_t$, for every $s \in I, t \in J$. Thus, $X$ and $Y$ are contained in a plane that is at the same time orthogonal to $X_s$ and  $X_t$ in $\mathbb{R}^4$. 

Let $N(s,t)$ be the unit normal field at $X(s,t)$ in $\mathbb{S}^3 \subset \mathbb{R}^4$ that is at the same time orthogonal to $X$, $X_s$ and $X_t$. Thus $N = a X + b Y$, with $a^2 + b^2 + 2 ab \langle X, Y \rangle  = 1$ and $\langle N, X \rangle = a + b \langle X, Y \rangle = 0$. Then $a = - b \langle X, Y \rangle $, which implies that 
$$ b^2( 1 + \langle X, Y \rangle^2) - 2 b^2 \langle X, Y \rangle^2 = b^2( 1 - \langle X, Y \rangle^2) = 1 . $$
Since we may choose $b = 1 / \sqrt{1 - \langle X, Y \rangle^2}$, and as $\langle X, Y \rangle  = \langle \alpha \cdot \beta , \alpha' \cdot \beta' \rangle = \langle\overline{\alpha} \cdot \alpha' , \beta \cdot \overline{\beta'} \rangle $, we get  
\begin{equation*}
N(s,t) =  \dfrac{\alpha'(s) \cdot \beta'(t) - \langle T_\alpha(s) , T_\beta(t) \rangle \ \alpha(s)\cdot\beta(t)}{\sqrt{1 - \langle T_\alpha(s) , T_\beta(t) \rangle^2}} .
\end{equation*}

Let $\nabla$ and $\tilde{\nabla}$ be the Levi-Civita connections in 
 $\mathbb{S}^3$ and $\mathbb{R}^4$ respectively. Since $p = X(s_0,t_0)$ is orthogonal to the surface $T_p X $ for every $s$ and $t$ (as the surface is contained in $\mathbb{S}^3 \subset\mathbb{R}^4$), and $\tilde{\nabla}$ is known to be equivalent to the usual differentiation, we have 
 $$ X_{ss} = \nabla_{X_s}X_s + \langle X_{ss} , X \rangle X, \ \ \ 
     X_{st} = \nabla_{X_s}X_t + \langle X_{st} , X \rangle X,  \ \ \ 
     X_{tt} = \nabla_{X_t}X_t + \langle X_{tt} , X \rangle X.  
$$
With these equations, we compute the coefficients of the second fundamental  form of the surface $X(s,t)$ as 
$$ e = \langle X_{ss}, N \rangle,  \ \ g = \langle X_{tt}, N \rangle , \ \ f = \langle X_{st}, N \rangle $$

Remembering that $\langle N , X \rangle = 0 $ and $ \langle \alpha \cdot \beta, \alpha'' \cdot \beta\rangle = -1 $, we begin computing the coefficients of the second fundamental form
$$  e = \langle N, \nabla_{X_s}X_s \rangle  =  \dfrac{\langle \alpha'\cdot\beta' - \langle \alpha'\cdot\beta' , \alpha\cdot\beta \rangle \alpha\cdot\beta, \alpha''\cdot \beta \rangle}{\sqrt{1 - \langle \alpha'\cdot\beta' , \alpha\cdot\beta \rangle^2}} = \dfrac{\langle \alpha'\cdot\beta', \alpha''\cdot \beta +  \alpha\cdot\beta \rangle}{\sqrt{1 - \langle \alpha'\cdot\beta' , \alpha\cdot\beta \rangle^2}}. $$
Here, if $\kappa_\alpha \equiv 0$, then $\alpha'' = -\alpha$ and $e = 0$. Symmetrically we have 
$$ g = \langle N, \nabla_{X_t}X_t \rangle = \dfrac{\langle \alpha'\cdot\beta', \alpha\cdot \beta'' +  \alpha\cdot\beta \rangle}{\sqrt{1 - \langle \alpha'\cdot\beta' , \alpha\cdot\beta \rangle^2}}. $$
Again, if $\kappa_\beta \equiv 0$, then $\beta'' = -\beta$ and $g = 0$. Also we have   
$$ f = \langle N, \nabla_{X_t}X_s \rangle = \dfrac{\langle \alpha'\cdot\beta' - \langle \alpha'\cdot\beta' , \alpha\cdot\beta \rangle \alpha\cdot\beta, \alpha' \cdot \beta' \rangle}{\sqrt{1 - \langle \alpha'\cdot\beta' , \alpha\cdot\beta \rangle^2}} = \sqrt{1 - \langle \alpha'\cdot\beta' , \alpha\cdot\beta \rangle^2}. $$
In case $\kappa_\alpha \not\equiv 0$ and $\kappa_\beta \not\equiv 0 $, it follows from the first system in \eqref{eq_curv_alpha}, Definition \ref{def:frames} and Proposition \ref{prop:frames} that the coefficients of the second fundamental form can be written as 
\begin{equation}\label{secondfundamnetalform}
    e = \dfrac{\kappa_\alpha \langle B_\alpha , T_\beta \rangle}{\sqrt{1 - \langle  T_\alpha , T_\beta \rangle }},  \ \ \  g  = -  \dfrac{\kappa_\beta \langle T_\alpha , B_\beta \rangle}{\sqrt{1 - \langle T_\alpha , T_\beta   \rangle }}, \ \ \  f = \sqrt{1 - \langle T_\alpha , T_\beta \rangle^2 } .
\end{equation}

We use the usual mean curvature formula to obtain  
\begin{equation*}
H = \dfrac{1}{2} \dfrac{eG - 2 fF + Eg }{EG - F^2} 
=  \dfrac{\kappa_{\alpha} \langle B_{\alpha} , T_\beta \rangle - \kappa_{\beta} \langle T_\alpha, B_{\beta}  \rangle - 2 \langle T_\alpha , T_\beta \rangle [ \langle T_\alpha , T_\beta \rangle^2 - 1] }{2[ 1 - \langle T_\alpha , T_\beta \rangle^2 ]^{3/2}} . 
\end{equation*}

In order to obtain the Gaussian curvature we must compute the extrinsic curvature by the classical equation  $K_{ext} = \dfrac{eg - f^2}{EG - F^2}$, that is 
\begin{equation}\label{curvextsimp}
 K_{ext} = - \frac{\kappa_\alpha\kappa_\beta\langle B_\alpha , T_\beta \rangle \langle T_\alpha, B_\beta \rangle }{1 - \langle T_\alpha , T_\beta \rangle^2} - 1 . 
\end{equation}
Thus, the Gaussian curvature is given by $K = K_{ext} + 1$. 

\end{proof}

\begin{remark} 
    From Theorem \ref{theo_prop_surface_general} we obtain some important equations that will be useful throughout this work. A translation surface is minimal if and only if 
    \begin{equation}\label{eqmincurvatsimp}
      \kappa_{\alpha} \langle B_\alpha , T_\beta  \rangle -  \kappa_{\beta} \langle  T_\alpha,  B_\beta \rangle =  2 \langle T_\alpha, T_\beta \rangle [ \langle T_\alpha, T_\beta \rangle^2 - 1 ] . 
    \end{equation}
    Furthermore, a translation surface is flat if and only if 
    \begin{equation}\label{eqflat}
    \kappa_\alpha\kappa_\beta\langle B_\alpha , T_\beta \rangle \langle T_\alpha, B_\beta \rangle = 0 .
    \end{equation}
\end{remark}

\section{Translation Surfaces with Constant Gaussian Curvature: Revisiting the Flat Case}\label{SectionFlat}

It is well-known from the Bianchi-Spivak construction \cite{Galvez, Spivack4} that every flat surface in $\mathbb{S}^3$ can be locally recovered as the quaternionic product of two curves in $\mathbb{S}^3$. In other words, every flat surface is locally a translation surface. On the other hand, the only translation surfaces in $\mathbb{S}^3$ with constant Gaussian curvature are the flat ones. This is the content of the following recent result:
 \begin{proposition}[\cite{XuHan}]\label{Theorem_Kconst_flat}
     Let $G$ be an $n$-dimensional ($n \geq 3$) Lie group with a bi-invariant metric, and $M$ be a translation surface in $G$ with constant Gaussian curvature, then $M$ must be flat.
 \end{proposition}
 This means that, since $\mathbb{S}^3$ is embedded in $ \mathbb{R}^4$ with the usual metric induced by the four-dimensional Euclidean space, which is a bi-invariant metric, the classification of translation surfaces with constant Gaussian curvature is reduced to the flat case.

 A direct consequence of this result is the non-existence of totally umbilic and totally geodesic translation surfaces in $\mathbb{S}^3$. In particular, the question of whether totally geodesic spheres are minimal translation surfaces is natural due to the fact that their analogues in $\mathbb{R}^3$, i.e. the planes, provide trivial examples of such surfaces. In this context, we present the following
\begin{theorem}\label{TeoTotalGeo}
     There is no totally umbilic surfaces or totally geodesic surfaces in $\mathbb{S}^3$ given as a translation surface.
\end{theorem} 
\begin{proof}
    Since a totally umbilic or totally geodesic surface has constant principal curvatures equal to $\lambda \in \mathbb{R}$, its Gaussian curvature $K$ is constant by the Gauss equation, which reads
    $$   K =   \lambda^2 + 1 .     $$
    However, by proposition \ref{Theorem_Kconst_flat}, $K$ must be zero, which contradicts the relation above.
\end{proof}

It is also a consequence of the Gauss Equation that flat surfaces in $\mathbb{S}^3$ have negative extrinsic curvature $K_{ext} = \lambda_1 \lambda_2$, where $\lambda_1$ and $\lambda_2$ denote the principal curvatures. In this context, we have the following well-known result:
\begin{theorem}[\cite{Galvez}]\label{TeoGalvez}
    Let $\Sigma$ be a surface and $\psi : \Sigma \to \mathbb{M}^3 (c)$ an immersion with negative constant extrinsic curvature $K_{ext}$ in a space form. Then the asymptotic curves of $\psi$ have constant torsion
    $\tau$, with $\tau^2 = -K_{ext}$ at points where the curvature of the curve does not vanish. Moreover, two
    asymptotic curves through a point have torsions of opposite signs if they have non-vanishing
    curvature at that point.
\end{theorem}

Such a result is particularly important when we recover a flat surface $\Sigma \subset \mathbb{S}^3$ from its asymptotic lines. Indeed, it is shown that the asymptotic lines within their family are congruent to one another  \cite[Proposition 3]{Galvez}, and the curves that generate the translation structure are precisely representatives of each class \cite[Theorem 9]{Galvez}. Our next result provides a kind of converse of these facts:
\begin{proposition}\label{remflat}
   Let $X : I \times J \rightarrow  \mathbb{S}^3$ ,  $X(s,t)=\alpha(s)\cdot \beta(t)$, be a translation surface. in $\mathbb{S}^3$. Suppose that $\kappa_\alpha \not\equiv 0 $ and $\alpha(s)\cdot \beta(t_0)$ is an asymptotic line for all $t_0 \in J \subset \mathbb{R}$. Then $\tau_\alpha = 1$ and either  $\kappa_\beta \equiv 0 $ or $\tau_\beta = - 1 $. If we have also $\kappa_\beta \not\equiv 0$ then $g = 0$, $\tau_\alpha \equiv 1 $ and $\tau_\beta \equiv - 1$. 
\end{proposition}
\begin{proof} 
    Suppose that $\alpha(s)\cdot \beta(t_0)$ is an asymptotic line of $X$ for all $t_0 \in J \subset \mathbb{R}$. This implies that $e = 0$ and, by Theorem \ref{theo_prop_surface_general} we have $\kappa_\alpha \langle B_\alpha, T_\alpha \rangle = 0$ and $K_{ext} \equiv -1$. It follows from Theorem \ref{TeoGalvez} that $\alpha$ has torsion $\tau = \pm 1$ where the curvature does not vanishes, since $\alpha$ is congruent to $\alpha \cdot \beta(t_0)$, for all $t_0 \in J$.
    
    Suppose now  that $\kappa_\alpha \not\equiv 0 $. In order to have $e = 0$, it is necessary that $ \langle B_\alpha , T_\beta \rangle = 0 $. Differentiating this equation with respect to $s$ yields $ - (\tau_\alpha - 1 ) \langle N_\alpha , T_\beta \rangle =  0 $. Assuming $\tau_\alpha = 1 $, we then differentiate once more to obtain
    $$  -( \tau_\alpha - 1 ) \langle (-\kappa_\alpha T_ \alpha + (\tau_\alpha -1 ) B_\alpha ), T_\beta \rangle = \kappa_\alpha( \tau_\alpha - 1 ) \langle T_ \alpha , T_\beta \rangle =  0 . $$
    Since $\kappa_\alpha \not\equiv 0$ we must have 
    $$  \langle  T_\alpha , T_\beta \rangle = \langle N_\alpha , T_\beta \rangle = \langle   B_\alpha , T_\beta \rangle  = 0, $$
    a contradiction as $  T_\alpha, \   N_\alpha , \  B_\alpha , \ T_\beta \in \mathcal{S} $. Hence $\tau_\alpha = 1 $.

    Now if $\tau_\alpha = 1$ then $B_\alpha = C$ a constant vector in $\mathcal{S}$. Since   $T_\beta \in \mathcal{S}$, if $\langle B_\alpha , T_\beta \rangle = 0 $, then $T_\beta$ is contained in the intersection of a two-dimensional plane that passes through the origin with $\mathcal{S}$, which means that either $B_\beta$ or $T_\beta$ is constant. Thus, either $\kappa_\beta \equiv 0$ or $\tau_\beta \equiv -1$.
    Since the conjugacy inverts the orientation, this implies the change of sign of the torsion (as this sign defines such orientation), which means that $\tau_\alpha = - \tau_{\overline{\alpha}}$.
    
    Suppose now also $\kappa_\beta \not\equiv 0 $, then $\tau_\alpha = -1 $ and $B_\alpha = C $. Thus
    $$ ef = \kappa_\alpha \langle C , T_\beta \rangle = 0.$$
    Differentiating with respect to $t$ gives $\kappa_\beta \langle C , N_\beta \rangle = 0$. Since $\kappa_\beta \not\equiv 0 $, then  $ C \perp N_\beta$ and $ C \perp T_\beta $. But $ C \perp N_\alpha$ and $ C \perp T_\alpha $ and also $T_\alpha, T_\beta, N_\alpha, N_\beta \in \mathcal{S}$, which means that they are all contained in a two dimensional plane in $\mathbb{R}^4$ that is, at the same time,  orthogonal to $e_1$ and $C$. As $B_\beta \in \mathcal{S}$, $B_\beta \perp T_\beta$ and $B_\beta \perp N_\beta$  we must have $B_\beta = \pm C$ and thus 
     $$\langle T_\alpha , B_\beta \rangle = \pm \langle T_\alpha , C \rangle  = 0 .$$ 
     Hence, $g= 0$. As shown before, we have $\tau_\beta = - 1$.
     
\end{proof}

Following \cite{Barros}, a curve $\gamma(s)$ in $\mathbb{S}^3$ is called a \textit{general helix} if there exists a Killing vector field $V(s)$ of constant length along $\gamma$ such that the angle between $V$ and $\gamma'$ is a nonzero constant the curve.  It is established in \cite[Theorem~3]{Barros}, that a curve $\gamma$ in $\mathbb{S}^3$ is a general helix if and only if either  $\tau \equiv 0$ and $\gamma$ is a curve in some unit 2-sphere $\mathbb{S}^2$ or there exists a constant $b$ such that $\tau = b\kappa \pm 1$. Therefore, a curve $\gamma \subset \mathbb{S}^3$ with constant curvature and torsion is a general helix since it satisfies the condition $\tau = b\kappa \pm 1$. In particular, when both $\kappa$ and $\tau$ are constant, we refer to such a curve as a \textit{proper helix}.

The definition of \textit{general helix} provides a nice geometric description of a translation surface generated by curves $\alpha$ and $\beta$ where $T_\alpha$ and $T_\beta$ makes a constant angle. Firstly, it follows from Theorem \ref{theo_prop_surface_general} that the metric components of a translation surface is given by $E=G=1$ and $F = \langle T_\alpha, T_\beta \rangle.$ Therefore, when $\langle T_\alpha, T_\beta \rangle$ is constant, such a surface is flat. We can go further and characterize the curves $\alpha$ and $\beta$ in this case:

\begin{theorem}\label{prop_FlatTorus}
    Let $X : I \times J \rightarrow  \mathbb{S}^3$ ,  $X(s,t)=\alpha(s)\cdot \beta(t)$, be a translation surface. If $ \langle T_\alpha , T_\beta\rangle = C$, then this surface is flat. Moreover, one has $\kappa_\beta \equiv 0$ and either $\kappa_\alpha \equiv 0$  or $\alpha$ is a general helix satisfying $ \tau_\alpha = \frac{C}{\eta} \kappa_\alpha + 1 $ with , $C, \eta \in \mathbb{R}$.  
\end{theorem}
\begin{proof}
Supposing that  $F = \langle T_\alpha , T_\beta \rangle  = C$ and knowing that $T_\alpha$ and $T_\beta$ are curves contained in $\mathcal{S} $, we may see this elements as curves in $\mathbb{S}^2 \subset \mathbb{R}^3$. Now, fixing $t_0$, the condition $\langle T_\alpha(s) , T_\beta(t_0) \rangle  = C$ for every $s$ implies that either  $T_\alpha$ is constant, hence, $\kappa_\alpha \equiv 0 $ or $T_\alpha$ is contained in a cone centered in $T_\beta(t_0)$. In the second case, the angle must remain constant if we choose $t_1 \ne t_0$, which implies by the geometry of the sphere $\mathbb{S}^2$ that $T_\beta$ is constant and thus $\kappa_\beta \equiv 0$. Symmetrically fixing $s_0$ we get that $T_\beta$ is constant or contained in a cone with center $T_\alpha$. 

Now, differentiating $\langle N_\alpha , T_\beta \rangle$ with respect to $s$ gives  
\begin{equation}\label{eqpropflatC}
  \kappa_\alpha \langle N_\alpha , T_\beta \rangle  = 0 .
\end{equation}
If $\kappa_\alpha \not\equiv 0$ we have $\langle N_\alpha , T_\beta \rangle  = 0$. Differentiating $\langle N_\alpha , T_\beta \rangle  = 0$ with respect to $s$ gives 
$$ - \kappa_\alpha \langle T_\alpha , T_\beta \rangle + ( \tau_\alpha - 1 )\langle B_\alpha , T_\beta \rangle = - \kappa_\alpha C + (\tau_\alpha - 1 )\langle B_\alpha , T_\beta \rangle  = 0 . $$

If $C = 0 $,  then $ (\tau_\alpha - 1) \langle B_\alpha , T_\beta \rangle  = 0  $. Thus either $\tau_\alpha \equiv 1 $  or $\langle B_\alpha , T_\beta \rangle  \equiv 0  $. If the second equality is true, then $T_\alpha \| T_\beta $ and $C = \pm 1 $, a contradiction. Hence $\tau_\alpha \equiv 1 $. 

If $C \ne 0 $, then either $\tau_\alpha \equiv 1 $ or $\tau_\alpha \not\equiv 1 $. If $\tau_\alpha \equiv 1 $ then $\kappa_\alpha C \equiv 0$, which implies that $C = 0$ or $\kappa_\alpha \equiv 0$, both contradictions. Thus,  $\tau_\alpha \not\equiv 1 $ and we have  
$$ \langle B_\alpha , T_\beta \rangle  = C \frac{\kappa_\alpha }{ \tau_\alpha - 1 } . $$
Differentiating with respect to $s$, and using equation \eqref{eqpropflatC}, gives 
$$ -(\tau_\alpha - 1 ) \langle N_\alpha , T_\beta \rangle  = \bigg( C \frac{\kappa_\alpha }{\tau_\alpha - 1}\bigg)' = 0, $$
which implies that $ C \kappa_\alpha =  \eta ( \tau_\alpha - 1 )$, $ \eta \in \mathbb{R} $. It is clear that $\eta \ne 0 $, otherwise $\kappa_\alpha C = 0$, which implies that $C = 0$ or $\kappa_\alpha = 0$, both contradictions. Hence we have 
$$ \frac{C}{\eta} \kappa_\alpha + 1 =  \tau_\alpha.$$
\end{proof}

As a consequence, we can extend this result to a general context, i.e., for any two curves in $\mathbb{S}^3$, in which  $ \langle \overline{\alpha} \cdot t_\alpha, \beta \cdot \overline{t_\beta} \rangle = C $.
\begin{corollary}
    Let $\alpha(s)$ and $\beta(t)$ be arc length curves in $\mathbb{S}^3$ such that $ \langle \overline{\alpha} \cdot t_\alpha, \beta \cdot \overline{t_\beta} \rangle = C$. Then $\kappa_\alpha \equiv 0$ and either $\kappa_\beta \equiv 0$  or $\beta$ is a general helix  with $ \tau_\beta = \frac{C}{\eta} \kappa_\beta  - 1 $, $C, \eta \in \mathbb{R}$.  
\end{corollary}

Still in the context of general helices, we now present the following result, which characterizes the behavior of the associated right frame of a general helix in $\mathbb{S}^3$ and will be useful in the next section.
\begin{proposition}\label{lemma_proper_helix}
    Let $\alpha$ be an arc length general helix in $\mathbb{S}^3$ with arc length parameter $s$. Then $T_\alpha$, $N_\alpha$ and $B_\alpha$ describe circles in $\mathcal{S}$.
\end{proposition}
\begin{proof}
    Let $\alpha(s)$ be a helix in $\mathbb{S}^3$, then  there exists a constant $b \in \mathbb{R}$ such that $\tau_\alpha = b \kappa_\alpha \pm 1  $. We may suppose without loss of generality that $\tau_\alpha = b \kappa_\alpha + 1  $.  Initially, if $\kappa_\alpha \equiv 0$ then $T_\alpha$ is constant and $N_\alpha$ and $B_\alpha$ are not defined. If $\tau_\alpha \equiv 1 $ then $B_\alpha$ is constant and $T_\alpha$ and $N_\alpha$ describe the great circle that is orthogonal to $B_\alpha$ and $e_1$. Thus, suppose from now on that $\kappa_\alpha \not\equiv 0$ and $\tau_\alpha \not\equiv 1 $ and consider the curve  $\hat{\alpha}(s) = T_\alpha(s)$. Since $\hat{\alpha}$ is in $\mathcal{S}$, then is immediate that $\tau_{\hat{\alpha}}$ vanishes. We compute  
    $$ \frac{d}{ds} \hat{\alpha}(s) = T'_\alpha(s) = \kappa_\alpha(s) N_\alpha(s) . $$ 
    Hence, for such a curve, consider the arc length parameter $\hat{s}(s) = \int_0^s  \kappa_\alpha (s) \ ds$.
   
    Hence $\frac{d}{d\hat{s}} s(\hat{s}) = 1/ \kappa_\alpha (\hat{s})$.
   
    Thus 
    $$ \frac{d}{d\hat{s}} \hat{\alpha}(\hat{s}) = \frac{1}{\kappa_\alpha(\hat{s})} [\kappa_\alpha(\hat{s}) N_\alpha(\hat{s})] = N_\alpha (\hat{s}). $$
    Since $\hat{s}$ is the arc length parameter of $\hat{\alpha}$, we conclude that $t_{\hat{\alpha}} (\hat{s})= N_\alpha (\hat{s})$. Now, we compute  
    $$ t_{\hat{\alpha}}' = \frac{d}{d \hat{s}}  N_\alpha(\hat{s})  = \frac{1}{\kappa_\alpha(\hat{s})} [ - \kappa_\alpha (\hat{s}) T_\alpha(\hat{s}) + (\tau_\alpha(\hat{s}) - 1 ) B_\alpha (\hat{s})] = \frac{\tau_\alpha(\hat{s}) - 1}{\kappa_\alpha(\hat{s})} B_\alpha (\hat{s}) -  \hat{\alpha}(\hat{s}). $$
    Since $B_\alpha \perp \hat{\alpha}$, $B_\alpha \perp t_{\hat{\alpha}}$ and using  equation \eqref{eq_curv_alpha}, we obtain $n_{\hat{\alpha}} = B_\alpha$ and $\kappa_{\hat{\alpha}}(\hat{s}) =  (\tau_\alpha(\hat{s}) - 1 )/ \kappa_\alpha(\hat{s}) $. Now, as $\tau_\alpha(\hat{s}) = b \kappa_\alpha (\hat{s}) + 1 $ we have 
    $$ \kappa_{\hat{\alpha}} =  \frac{(\tau_\alpha(\hat{s}) - 1 )}{ \kappa_\alpha(\hat{s})} =  \frac{b \kappa_\alpha(\hat{s})}{\kappa_\alpha(\hat{s})}  = b . $$
    Since $b$ is constant by hypothesis, we conclude that $T_\alpha$ describes a circle in $\mathcal{S}$.  
    
    We remember that in $\mathbb{S}^3$, a small circle is always contained in a small sphere, that is, for some $v \in \mathbb{S}^3$ we describe a small sphere as $\mathcal{S}_v = \{ w \in \mathbb{S}^3 : \ w \perp v   \} $. Thus, for some $u \in \mathcal{S}_v$ and a constant $\theta\in \mathbb{R}$, a small circle $\mathcal{C}_{v,u,\theta} = \{ w \in \mathcal{S}_v : \langle w, u\rangle = \cos(\theta)  \} $ has pole (or spherical center) given by $v$ and $u$. Since $T_\alpha$ describes a circle in $\mathcal{S}$, then suppose that for some $u \in \mathcal{S}$, it describes the circle $\mathcal{C}_{e_1,u,\theta} = \{ w \in \mathcal{S} : \langle w, u\rangle = \cos(\theta)  \} $. Thus $\langle T_\alpha , u \rangle = \cos(\theta) $. Differentiating with respect to $s$ gives $ \kappa_\alpha \langle N_\alpha , u \rangle = 0 $. As $\kappa_\alpha \not\equiv 0$, it follows that  $t_{\tilde{\alpha}} = N_\alpha$ describes the great circle that is orthogonal to $u$ and $e_1$. Now, differentiating $\langle N_\alpha, u \rangle $ again with respect to $s$ gives 
    $$ - \kappa_\alpha \langle T_\alpha, u \rangle + (\tau_\alpha - 1 )\langle B_\alpha, u \rangle = 0 . $$
    Hence 
    $$ \langle B_\alpha, u \rangle = \frac{\kappa_\alpha}{\tau_\alpha - 1} \cos(\theta) =  \frac{1}{b} \cos(\theta), $$
    which implies that $\langle B_\alpha, u \rangle $ is constant. Thus, $B_\alpha$ describes a circle in $\mathcal{S}$ and $T_\alpha$, $N_\alpha$ and $B_\alpha$ have the same pole.

\end{proof}

\section{On CMC Translation Surfaces in \texorpdfstring{$\mathbb{S}^3$}{TEXT}}\label{SectionCMC}

We begin this section with the following example, which presents the CMC Clifford tori as translation surfaces with generating curves given by great circles.

\begin{example}\label{ex_flat_torus}
    It is well known that some classical examples of flat surfaces in $\mathbb{S}^3$ are the so-called Clifford torus $C_{R_1,R_2}$, given by
    $$ C_{R_1,R_2} = \bigg\{ (x_1, x_2, x_3, x_4) \in \mathbb{R}^4 \ \ \bigg| \ \  x_1^2 + x_2^2 = R_1^2 , \ \ x_3^2 + x_4^2 = R_2^2 \bigg\} .$$
    One can parameterize this kind of surface as 
    $$X(s,t) = (R_1 \cos(s+t), R_1 \sin(s+t), R_2 \cos(s-t), R_2 \sin(s-t) ), \ \textnormal{ with } R_1^2 + R_2^2 = 1 .$$
    Now, consider the following linear map that happens to be a rotation in $\mathbb{R}^4$
    $$M_{R_1,R_2} = \begin{bmatrix}
        R_1 & 0 & R_2 & 0 \\
        0 & R_1 & 0 & R_2 \\
        R_2 & 0 & -R_1 & 0 \\
        0 & R_2 & 0 & -R_1  
    \end{bmatrix}, \ \ \ \textnormal{Det}(M_{R_1,R_2}) = 1 . $$
    Thus, we have 
    $$ M_{R_1,R_2} (X(s,t)) = 
    \begin{bmatrix}
        \cos(s)\cos(t) - (R_1^2  - R_2^2) \sin(s)\sin(t) \\
        \sin(s)\cos(t) + (R_1^2  - R_2^2) \cos(s)\sin(t) \\
        - 2 R_1 R_2 \sin(s)\sin(t) \\
        2 R_1 R_2 \cos(s)\sin(t) 
    \end{bmatrix} . $$
    Such parameterization implies that $M_{R_1,R_2} (X(s,t)) = \alpha(s) \cdot \beta(t)$, where 
    $$ \alpha(s) = (\cos(s),\sin(s),0,0), \  \ \ \beta(t) = (\cos(t),(R_1^2 - R_2^2)\sin(t), 0, 2R_1R_2  \sin(t)).  $$
    Hence, $  \langle T_\alpha, T_\beta \rangle = - (R_1^2 - R_2^2) $. Since $\kappa_\alpha = \kappa_\beta = 0$, by Theorem \ref{theo_prop_surface_general} we have 
    $$H = \dfrac{- \langle T_\alpha , T_\beta \rangle }{\sqrt{1 - \langle T_\alpha , T_\beta \rangle^2}} = \dfrac{(R_1^2 - R_2^2)}{\sqrt{1 - (R_1^2 - R_2^2)^2}} . $$
    Since $R_1$ and $R_2$ are constant, then $\alpha(s) \cdot \beta(t) $ is a CMC Surface. Moreover, when $R_1 = R_2 = 1/\sqrt{2} $, we have the minimal and flat Clifford torus (See Figure \ref{fig2:sub1}).
\end{example}

\begin{figure}[h]
\centering
  \includegraphics[width=.4\linewidth]{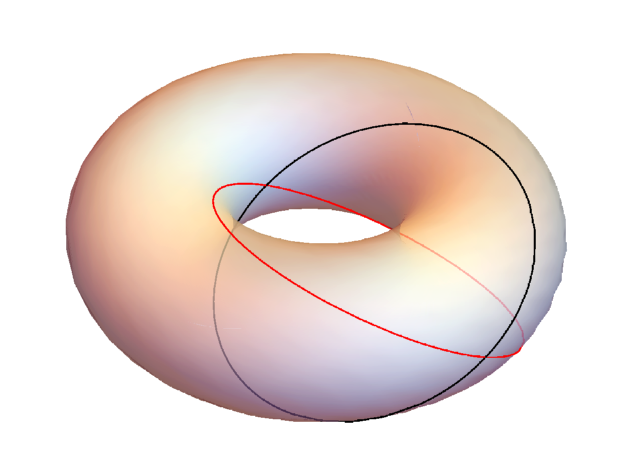}
  \includegraphics[width=.37\linewidth]{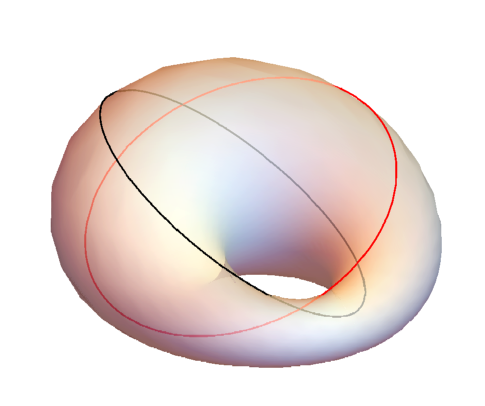}
  \caption{Illustration of two Clifford tori with mean curvatures $H =0$ and $H = 1 / \sqrt{3}$, respectively. The generating curves $\alpha$ and $\beta$ are highlighted in red and black, respectively.}
  \label{fig2:sub1}
\end{figure}

The first result of the section shows that the examples above are unique in the sense of generating curves given by great circles.

\begin{theorem}\label{TeoCMCbothcurvatzero}
Let $X : I \times J \rightarrow  \mathbb{S}^3$,  $X(s,t)=\alpha(s)\cdot \beta(t)$, be a translation surface. If $\kappa_\alpha = \kappa_\beta = 0$, then it is a CMC Clifford torus. Moreover, we have $\langle T_\alpha , T_\beta \rangle  = C \in (-1,1)$  and the mean curvature is given by
$$H = \dfrac{- C }{\sqrt{1 - C ^2}} . $$
\end{theorem}
\begin{proof}
    If $\kappa_\alpha = \kappa_\beta = 0$, then $T_\alpha $ and $T_\beta $ are constant vectors ans also, this surface is flat. Thus, $ \langle T_\alpha , T_\beta \rangle = C \in (-1,1) $ and equation \eqref{Meancurvature} becomes 
     $$H = \dfrac{- \langle T_\alpha , T_\beta \rangle }{\sqrt{1 - \langle T_\alpha , T_\beta \rangle^2}} = \dfrac{- C }{\sqrt{1 - C ^2}} . $$
    
    Moreover, from the proof of \cite[Proposition~3.4]{Hoffman}, we know that such surfaces must be a standard product of circles, $\mathbb{S}^1(r)\times \mathbb{S}^1(\rho)$, and thus CMC Clifford tori.
\end{proof}
We can again use the argument based on \cite[Proposition~3.4]{Hoffman} to show that CMC flat surfaces where $\langle T_\alpha, T_\beta\rangle$ is constant or one of the generating curves is a great circle must necessarily be a CMC Clifford torus.
\begin{theorem}\label{TeoCMCcurvatzero} 
    Let $X : I \times J \rightarrow  \mathbb{S}^3$,  $X(s,t)=\alpha(s)\cdot \beta(t)$, be a CMC translation surface. If $ F = \langle T_\alpha, T_\beta\rangle = C \in (-1,1) $  or $\kappa_\alpha = 0$ (symmetrically $\kappa_\beta = 0$), then this surface is a CMC Clifford Torus.
\end{theorem}
\begin{proof}
If $ F = \langle T_\alpha, T_\beta\rangle = C \in \mathbb{R}$, then this surface is flat. By Theorem \ref{theo_prop_surface_general}, the same occurs when $\kappa_\alpha \equiv 0$ (or $\kappa_\beta \equiv 0$). Again, from the proof of \cite[Proposition~3.4]{Hoffman}, this surface is a CMC Clifford torus.

\end{proof}

\section{Correspondence between Translation Surfaces in \texorpdfstring{$\mathbb{S}^3$}{TEXT} and \texorpdfstring{$\mathbb{R}^3$}{TEXT} }\label{SectionR3}

In this section, we present a result that establishes a connection between translation surfaces in $\mathbb{S}^3$ and translation surfaces in $\mathbb{R}^3$. The objective here is to understand the relationship between these surfaces and analyze them through both their intrinsic and extrinsic geometry. We then present some applications, drawing on examples and results from \cite{LopezHasanis,LopezPerdomo}. Accordingly, we state the following  
\begin{theorem}\label{theoS3R3}
    Let  $M \subset \mathbb{S}^3$ be a translation surface generated by curves  $\alpha$ and $\beta$ with curvatures $ \kappa_\alpha$, $ \kappa_\beta$ and, when $\kappa_\alpha \not\equiv 0$ and $\kappa_\beta \not\equiv 0$,  torsions $\tau_\alpha$ and  $\tau_\beta$. Then this surface is locally isometric to a translation surface $\tilde{M} \subset \mathbb{R}^3$, generated by curves $\tilde{\alpha}$ and $\tilde{\beta}$, whose curvatures  and torsions satisfy $\tilde{\kappa}_\alpha = \kappa_\alpha  $, $\tilde{\kappa}_\beta = \kappa_\beta  $, $\tilde{\tau}_\alpha = ( \tau_\alpha - 1 )$ and  $\tilde{\tau}_\beta = (\tau_\beta + 1)$, respectively. The reciprocal identification also holds. Moreover, the  mean curvatures $\tilde{H}$ and $H$, of $\tilde{M}$ and $M$, respectively, satisfy 
    $$ \tilde{H} = H + \frac{\langle T_\alpha , T_\beta \rangle}{\sqrt{1 - \langle T_\alpha , T_\beta \rangle^2 }}. $$
\end{theorem}
\begin{proof}
Suppose initially that $\kappa_\alpha = \kappa_\beta \equiv 0 $, then we write $ \overline{\alpha} \cdot t_\alpha = T_\alpha $, which is constant. We associate to this a curve $\tilde{\alpha}$ in $\mathbb{R}^3$, which is a straight line in the direction of $T_\alpha$. By symmetry, we proceed similarly for $\beta$ and $\tilde{\beta}$. Now, consider a translation surface in $\mathbb{R}^3$ defined by $\Psi(s,t) = \Tilde{\alpha}(s) + \Tilde{\beta}(t)  $. We have that 
$$\begin{array}{ccccccccc}
     \tilde{E} & = & \langle\Psi_s(s,t) , \Psi_s(s,t) \rangle & = & \langle \tilde{\alpha}'(s) , \tilde{\alpha}'(s) \rangle & = & \langle T_\alpha , T_\alpha \rangle & = & 1 , \\
     \tilde{G} & = & \langle\Psi_t(s,t) , \Psi_t(s,t) \rangle & = & \langle \tilde{\beta}'(t) , \tilde{\beta}'(t) \rangle & = & \langle T_\beta , T_\beta \rangle & = & 1 , \\
     \tilde{F} & = & \langle \Psi_s(s,t) , \Psi_t(s,t) \rangle & = & \langle \tilde{\alpha}'(s) , \tilde{\beta}'(t) \rangle & = & \langle T_\alpha , T_\beta \rangle . &  & \\
\end{array}$$ 
By  \cite{LopezHasanis}, we know also that 
$$ \tilde{N}(s,t) = \frac{T_\alpha \times T_\beta}{\sqrt{1 - \langle T_\alpha , T_\beta \rangle^2}} . $$
Thus 
$$\begin{array}{rllllll}
     \tilde{e} & = & \langle\Psi_ss(s,t) , N(s,t) \rangle & = & \langle \tilde{\alpha}'' , N \rangle & = & 0 ,  \\
     \tilde{g} & = & \langle\Psi_tt(s,t) , N(s,t) \rangle & = & \langle \tilde{\beta}'' , N \rangle & = & 0 , \\
     \tilde{f} & = & \langle\Psi_st(s,t) , N(s,t) \rangle & = & \langle 0 , N \rangle & = & 0 . \\
\end{array}$$
By Theorem \ref{theo_prop_surface_general}, we have $\tilde{E} = E = \tilde{G} = G = 1 $, $\tilde{F} = F$ and $\tilde{e} = e = \tilde{g} = g = 0$. Consequently  
$$ \tilde{K} = - \frac{\tilde{e}\tilde{g} - \tilde{f}}{\tilde{E}\tilde{G} - \tilde{F}^2} = 0 =  K .  $$
where $\tilde{K}$ and $K$ are the Gaussian curvatures in $\mathbb{R}^3$ and $\mathbb{S}^3$ respectively. Moreover, we have 
$$  \tilde{H} = \frac{\tilde{e}\tilde{G} - 2\tilde{f}\tilde{F} + \tilde{E}\tilde{g}}{2(\tilde{E}\tilde{G} - \tilde{F}^2)} =  H + \frac{\langle T_\alpha , T_\beta \rangle}{\sqrt{1 - \langle T_\alpha , T_\beta \rangle^2 }} = 0,  $$
where $\tilde{H}$ and $H$ are the mean curvatures in $\mathbb{R}^3$ and $\mathbb{S}^3$ respectively.

Suppose now that $\kappa_\alpha  \equiv 0 $ and $\kappa_\beta \not\equiv 0$ (or symmetrically $\kappa \equiv 0$, $\kappa_\alpha \not\equiv 0$). With the vectors $T_\beta$, $N_\beta$, $B_\beta$ and equations of Frenet-frame kind 
$$\begin{array}{rcl}
     T_\beta' & = & \kappa_\beta N_\beta ,  \\
     N_\beta' & = & -\kappa_\beta T_\beta + ( \tau_\beta + 1 ) B_\beta , \\
     B_\alpha' & = & -( \tau_\beta + 1 ) N_\beta . 
\end{array} $$
We can associate a unique arc length  curve $\Tilde{\beta}$ in $\mathbb{R}^3$, up to rigid motion, whose curvature and torsion satisfy $\tilde{\kappa}_\beta = \kappa_\beta$ and $\tilde{\tau}_\beta = ( \tau_\beta + 1 )$, respectively, and whose Frenet frame is given by $\{ T_\beta, N_\beta, B_\beta \}$.
 
As before, we associate to $\alpha$ a curve $\tilde{\alpha}$ in $\mathbb{R}^3$ that is a straight line parallel to $T_\alpha$. Consider now a translation surface in $\mathbb{R}^3$ defined by $\Psi(s,t) = \Tilde{\alpha}(s) + \Tilde{\beta}(t)  $. We have that 
$$ \tilde{E} =  \langle T_\alpha , T_\alpha \rangle  = 1 ,  \ \  \tilde{G} = \langle T_\beta , T_\beta \rangle = 1,  \ \  \tilde{F} =  \langle T_\alpha , T_\beta \rangle . $$ 
Again, $ \tilde{N}(s,t) = T_\alpha \times T_\beta / \sqrt{1 - \langle T_\alpha , T_\beta \rangle^2}$. Thus 
$$\tilde{e} = \langle \tilde{\alpha}'' , N \rangle =   0, \ \  \tilde{g} = \langle \tilde{\beta}'' , N \rangle = - \dfrac{\kappa_\beta \langle T_\alpha , B_\beta \rangle }{\sqrt{1 - \langle T_\alpha , T_\beta \rangle^2} }, \ \ \tilde{f} = \langle 0 , N \rangle =  0 . $$
Once again, by Theorem \eqref{theo_prop_surface_general}, we conclude that $\tilde{E} = E = \tilde{G} = G = 1 $, $\tilde{F} = F$ and $\tilde{e} = e = 0 $, $\tilde{g} = g$. Hence, $ \tilde{K} =  0 =  K  $ and also  
$$  \tilde{H} =  \frac{- \kappa_\beta  \langle T_\alpha , B_\alpha \rangle}{2(1 - \langle T_\alpha , T_\beta \rangle^2)^{3/2}} - \frac{\langle T_\alpha , T_\beta \rangle}{\sqrt{1 - \langle T_\alpha , T_\beta \rangle^2 }} +   \frac{\langle T_\alpha , T_\beta \rangle}{\sqrt{1 - \langle T_\alpha , T_\beta \rangle^2 }} =  H + \frac{\langle T_\alpha , T_\beta \rangle}{\sqrt{1 - \langle T_\alpha , T_\beta \rangle^2 }} . $$

From now on, suppose that $\kappa_\alpha\not\equiv 0$ and  $ \kappa_\beta \not\equiv 0 $.   Consider now a translation surface in $\mathbb{R}^3$ defined by $\Psi(s,t) = \Tilde{\alpha}(s) + \Tilde{\beta}(t)  $. We have that 
$$ \tilde{E} = \langle T_\alpha , T_\alpha \rangle = 1, \ \ \tilde{G} = \langle T_\beta , T_\beta \rangle = 1, \ \ \tilde{F} = \langle T_\alpha , T_\beta \rangle . $$
One more time, by \cite{LopezHasanis}, we know  that $ \tilde{N}(s,t) = (T_\alpha \times T_\beta) / \sqrt{1 - \langle T_\alpha , T_\beta \rangle^2} $. Thus  
$$\tilde{e} =  \langle \tilde{\alpha}'' , N \rangle =  \dfrac{\kappa_\alpha \langle B_\alpha , T_\alpha \rangle }{\sqrt{1 - \langle T_\alpha , T_\beta \rangle^2}}, \ \ \tilde{g} = \langle \tilde{\beta}'' , N \rangle =  - \dfrac{\kappa_\beta \langle T_\alpha , B_\beta \rangle }{\sqrt{1 - \langle T_\alpha , T_\beta \rangle^2} },  \ \   \tilde{f} = \langle 0 , N \rangle = 0 . $$
Thus $\tilde{E} = E = \tilde{G} = G = 1 $, $\tilde{F} = F$ and $\tilde{e} = e$ and $\tilde{g} = g$. Hence 
$$ \tilde{K} = - \frac{\kappa_\alpha \kappa_\beta \langle B_\alpha , T_\alpha \rangle \langle T_\alpha , B_\beta \rangle}{1 - \langle T_\alpha , T_\beta \rangle^2} = K ,  $$
and   
$$  \tilde{H} = \frac{\kappa_\alpha \langle B_\alpha , T_\alpha \rangle - \kappa_\beta  \langle T_\alpha , B_\alpha \rangle}{2(1 - \langle T_\alpha , T_\beta \rangle^2)^{3/2}}  = H + \frac{\langle T_\alpha , T_\beta \rangle}{\sqrt{1 - \langle T_\alpha , T_\beta \rangle^2 }} . $$

Now, consider a curve $\tilde{\alpha}$ in $\mathbb{R}^3$. If $\tilde{\alpha}$ is a straight line, we correspond it with a geodesic circle in $\mathbb{S}^3$. Otherwise, let $\{T_{\tilde{\alpha}}, N_{\tilde{\alpha}}, B_{\tilde{\alpha}} \}$ denote its Frenet frame, $\tilde{\kappa}_\alpha$ and $\tilde{\tau}_\alpha$ its curvature and torsion, respectively. Let $\alpha$ be the unique curve in $\mathbb{S}^3$, up to rigid motion, with curvature $\kappa_\alpha = \tilde{\kappa}_\alpha$ and torsion  $\tau_\alpha = \tilde{\tau}_\alpha + 1 $. We know that $\{T_\alpha, N_\alpha, B_\alpha\}$ satisfy \eqref{propertiesalpha}, which correspond to the Frenet formulas for $\tilde{\alpha}$ in $\mathbb{R}^3$. Now, using Theorem  \eqref{theo_prop_surface_general} and \cite{LopezHasanis}, we conclude the result. 
\end{proof}

As applications of this Theorem, we have the following
\begin{corollary}
    Let $\tilde{M} \subset \mathbb{R}^3$ and $M \subset \mathbb{S}^3$ be translation surfaces in the conditions of Theorem \ref{theoS3R3}. If they both are CMC, then they are also flat, $M$ is a CMC flat torus and  $\tilde{M}$ is a plane or a cylinder. 
\end{corollary}
and 
\begin{example}
    Let $M \subset \mathbb{S}^3$ be a translation surface such that $ \langle T_\alpha , T_\alpha \rangle = C $. Then, by Theorem \ref{TeoCMCcurvatzero}, this surface is a flat CMC torus, and thus the associated surface $\tilde{M} \subset \mathbb{R}^3$ is also flat and CMC. Therefore, $\tilde{M}$ must be a plane or a cylinder. Moreover, by Proposition \ref{prop_FlatTorus}, the condition $\langle T_\alpha, T_\beta \rangle = C$ implies that either $\kappa_\alpha \equiv \kappa_\beta \equiv 0 $, or $\kappa_\beta \equiv 0 $ and $\alpha$ is a general helix. In the case $\kappa_\alpha \equiv \kappa_\beta \equiv 0 $, $\tilde{M}$ is a plane. If $\alpha$ is a helix, then by the proof of  Proposition \ref{prop_FlatTorus}, $\langle B_\alpha, T_\beta\rangle$ is constant. Thus, by Theorem \eqref{theo_prop_surface_general}, if $M$ is CMC, then $\kappa_\alpha$ is constant, which implies that $\tau_\alpha$ is constant as well. Hence, $\tilde{\alpha}$ is a helix in $\mathbb{R}^3$. 
\end{example}
    
The next ones are motivated by the works  \cite{LopezHasanis} and \cite{LopezPerdomo}. In particular, they approach the case of circular helices in $\mathbb{R}^3$. 
\begin{corollary}\label{corolary_JP}
    Let $\tilde{M} \subset \mathbb{R}^3$ and $M \subset \mathbb{S}^3$ be translation surfaces in the conditions of Theorem \ref{theoS3R3}. If $\kappa_\alpha \ne 0 $ and $|\tau_\alpha| \ne 1 $ are constant then $\alpha(s) \cdot \alpha(t)$ is not minimal in $\mathbb{S}^3$. 
\end{corollary}

\begin{proof}
     Following \cite[Theorem~3.2]{LopezHasanis}, let $\alpha$ in $\mathbb{R}^3$ be an arc length curve with constant curvature and torsion. Then a translation surface $\tilde{M} \subset \mathbb{R}^3$, locally parameterized as $\tilde{X}(s,t) = \alpha(s) + \alpha(t)$, is minimal if and only if is a helicoid. By Theorem \ref{theoS3R3}, we conclude that $M \subset \mathbb{S}^3$ cannot be minimal.
\end{proof}

\begin{remark}
    For corollary \ref{corolary_JP}, it is important to keep the regularity condition, that is, $\langle T_\alpha, T_\beta \rangle \ne 1 $. 
\end{remark}

About curves with constant curvature and torsion, we have
\begin{corollary}
    If $\tilde{M} \subset \mathbb{R}^3$ is minimal and one of the curves is a circular helix. Then  the surface $M \subset \mathbb{S}^3$ is neither CMC nor flat, and its mean curvature is given by  
    $$ H = - \frac{\langle T_\alpha , T_\beta \rangle}{\sqrt{1 - \langle T_\alpha , T_\beta \rangle^2 }} .  $$
\end{corollary}
\begin{proof}
     By \cite[Theorem~3.2]{LopezHasanis}, if $\tilde{M} \subset \mathbb{R}^3$ is a minimal translation surface and one of the generating curves has constant curvature and torsion, then the other curve is a congruent curve with the same curvature and torsion, and $\tilde{M}$ is the helicoid.  By Theorem \ref{theoS3R3}, the surface $M \subset \mathbb{S}^3$ is neither CMC nor flat, and its mean curvature is given by  
    $$ H = - \frac{\langle T_\alpha , T_\beta \rangle}{\sqrt{1 - \langle T_\alpha , T_\beta \rangle^2 }} . $$    
\end{proof}

\section{Results on Minimal Translation Surfaces}\label{section_minimal}

In this section we deal with minimal translation surfaces. The first Theorem of this section can be viewed as a generalization of Corollary \ref{corolary_JP}.

\begin{theorem}\label{theokconsttau-1}
   There are no minimal translation surfaces $X : I \times J \rightarrow  \mathbb{S}^3$,  $X(s,t)=\alpha(s)\cdot \beta(t)$ with $\kappa_\alpha \equiv C \in \mathbb{R}$, $C \ne 0 $ and $\tau_\alpha = 1 $.
\end{theorem}
\begin{proof}
Differentiating equation \eqref{eqmincurvatsimp} with respect to $s$ gives  
\begin{equation}\label{eqmincurvatsimpder}
    \kappa'_{\alpha} \langle B_\alpha , T_\beta  \rangle - \kappa_{\alpha} (\tau_\alpha - 1) \langle N_\alpha ,  T_\beta \rangle - \kappa_{\beta}\kappa_\alpha \langle N_\alpha ,  B_\beta \rangle = 2 \kappa_\alpha \langle N_\alpha , T_\beta \rangle [ 3\langle T_\alpha, T_\beta \rangle^2 - 1 ].
\end{equation}
If $\tau_\alpha = 1$ and $\kappa_\alpha \equiv C \ne 0$, $C \in \mathbb{R}$, then equation \eqref{eqmincurvatsimpder} becomes
$$ - \kappa_\beta \langle N_\alpha ,  B_\beta \rangle = 2  \langle N_\alpha , T_\beta \rangle [ 3\langle T_\alpha, T_\beta \rangle^2 - 1 ]. $$ 
Differentiating again with respect to $s$ and using \eqref{eqmincurvatsimp}, we have  
$$ \kappa_{\alpha} \langle B_\alpha , T_\beta  \rangle = - 2 \langle T_\alpha , T_\beta \rangle [ 2\langle T_\alpha, T_\beta \rangle^2 - 6 \langle N_\alpha , T_\beta \rangle^2 ]. $$
We differentiate again with respect to $s$ and simplify to obtain 
$$  12 \kappa_\alpha \langle N_\alpha , T_\beta \rangle^3 +  36 \kappa_\alpha \langle T_\alpha , T_\beta \rangle^2 \langle N_\alpha, T_\beta \rangle  = 0. $$
Suppose that $\langle N_\alpha, T_\beta \rangle \ne 0 $. Then we have $  \langle N_\alpha , T_\beta \rangle^2 = -3 \langle T_\alpha , T_\beta \rangle^2   $, and differentiating  with respect to $s$ gives 
$$ -2 \kappa_\alpha \langle T_\alpha , T_\beta \rangle \langle N_\alpha , T_\beta \rangle = -6 \kappa_\alpha \langle N_\alpha , T_\beta \rangle \langle T_\alpha , T_\beta \rangle , $$
which implies that $\langle T_\alpha, T_\beta \rangle = \langle N_\alpha, T_\beta \rangle = 0 $, a contradiction.  

Suppose now that $\langle N_\alpha , T_\beta \rangle = 0 $. Differentiating with respect to $s$ yields
$ \kappa_\alpha \langle T_\alpha , T_\beta \rangle = 0$. Since $\kappa_\alpha \ne 0$ by hypothesis, it follows that  $\langle T_\alpha , T_\beta \rangle = 0$. Moreover, 
since $T_\beta \perp T_\alpha $ and  $T_\beta \perp N_\alpha $, then $T_\beta = \pm B_\alpha $ which is a constant vector. Thus, $\kappa_\beta \equiv 0 $. 

Now, returning to equation \eqref{Meancurvature} we obtain 
$$ \kappa_\alpha \langle B_\alpha , T_\beta \rangle = 0. $$
Thus $\langle B_\alpha , T_\beta \rangle = 0 $, a contradiction.

\end{proof}

When $\tau_\alpha$ is any constant, not necessarily equal to $1$, we impose conditions on the curve $\beta$ to obtain the following non-existence result:

\begin{theorem}\label{Theo_min_2helix}
   The are no  minimal translation surface $X : I \times J \rightarrow  \mathbb{S}^3$,  $X(s,t)=\alpha(s)\cdot \beta(t)$, where $\kappa_\alpha \ne 0 $, $\kappa_\beta \ne 0 $, $\tau_\alpha$  and $\tau_\beta $ are constant.
\end{theorem}

\begin{proof}
Since the case where $\tau_\alpha = 1$ and $\tau_\beta = - 1 $ was treated in Theorem \ref{theokconsttau-1}, we assume from now on that $ \tau_\alpha \ne 1$ and  $\tau_\beta \ne - 1$. Let $\alpha(s)$ and $\beta(t)$ be arc length curves that are also proper helices in $\mathbb{S}^3$. By lemma \ref{lemma_proper_helix}, the curves $T_\alpha(s)$, $N_\alpha(s)$, $ B_\alpha(s) $, $T_\beta(t)$, $N_\beta(t)$ and $B_\beta(t)$ all trace circles in $\mathcal{S}$. Moreover, assuming without loss of generality that $(1-\tau_\alpha) > 0 $ and $- (1+ \tau_\beta) > 0$, the curves $\tilde{\alpha}(s) = B_\alpha(s)$ and $\tilde{\beta}(t) = B_\beta(t)$ have curvatures $\tilde{\kappa}_\alpha = \kappa_\alpha / (1 - \tau_\alpha) $ and $\tilde{\kappa}_\beta = -\kappa_\beta / (1 +  \tau_\beta) $, respectively. 

Now, since $T_\alpha, N_\alpha, B_\alpha, T_\beta, N_\beta, B_\beta  \in \mathcal{S}$, to compute $\langle T_\alpha, T_\beta\rangle $, $\langle B_\alpha, T_\beta\rangle $ and $\langle T_\alpha, B_\beta\rangle$, we reduce the problem to one in $\mathbb{S}^2 \subset \mathbb{R}^3$.  This reduction is valid because the values of these inner products depend only on the angles between the vectors involved, not on their specific positions
 
(see Figure \ref{fig_2circ}). 
\begin{figure}[h]
\centering
\includegraphics[width=0.3\linewidth]{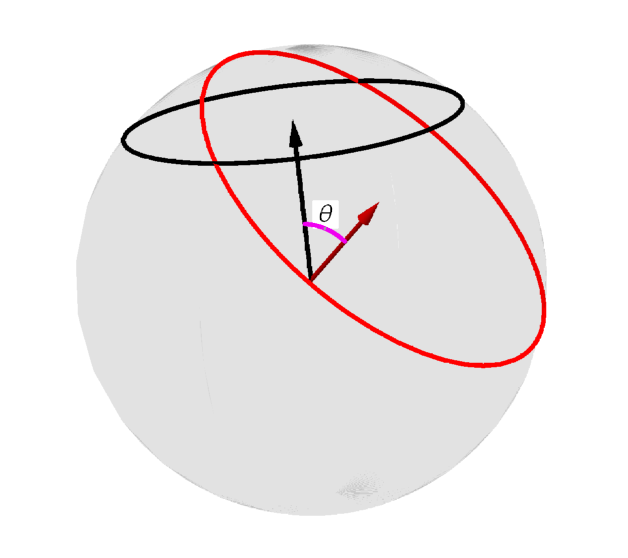}
\caption{\label{fig_2circ} Illustration of the relative position of the circles in Theorem \ref{Theo_min_2helix}.}
\end{figure} 

Thus, let $x(s) = (1 - \tau_\alpha) s $ and $y(t) = -(1 + \tau_\beta) t $ be arc lengths parameters for $\tilde{\alpha}$ and $\tilde{\beta}$, respectively. We can then choose convenient coordinates and parameterize these curves as follows 
$$\begin{array}{rl}
    \tilde{\alpha}(x) =  B_\alpha (x) & =  \left(Q \cos\left(\frac{1}{Q} x \right), Q \sin\left(\frac{1}{Q} x \right), P \right) ,  \\ 
    \tilde{\beta}(y) =  B_\beta (y) & = \left( S \cos\left(\frac{1}{S} y \right), S \cos \theta \sin\left(\frac{1}{S} y \right)  - R \sin \theta, S \sin \theta \sin\left(\frac{1}{S} y \right) + R \cos \theta  \right) .
\end{array}$$
where $0 < \theta < \pi$ is a constant angle and $P,Q, R,S > 0 $,  $P^2 + Q^2 = S^2 + R^2 = 1 $. 

Now, in order to better describe the curvatures of the generating curves, we differentiate $B_\alpha$ with respect to $x$ to obtain
$$  B_\alpha' =  N_\alpha  = \bigg( - \sin\bigg(\frac{1}{Q} x \bigg), \cos\bigg(\frac{1}{Q} x \bigg) , 0 \bigg) . $$ 
A second differentiation with respect to $x$ gives 
$$  N_\alpha' =  - \frac{1}{Q} \bigg( \cos\bigg(\frac{1}{Q} x \bigg), \sin\bigg(\frac{1}{Q} x \bigg) , 0 \bigg) . $$ 
Since $N_\alpha' = \tilde{\kappa}_\alpha n_{\tilde{\alpha}} - \tilde{\alpha} $, we have $  \langle N'_\alpha , N'_\alpha \rangle =  \tilde{\kappa}_\alpha^2 +  1 = 1/Q^2 $. Thus 
$$  \tilde{\kappa}_\alpha^2 =  \frac{1- Q^2}{Q^2} = \frac{P^2}{Q^2} .  $$
Symmetrically, $\tilde{\kappa}_\beta^2 = R^2/S^2 $ and 
$$  N_\beta =  \bigg( - \sin\bigg(\frac{1}{S} y \bigg), \cos(\theta) \cos \bigg(\frac{1}{S} y \bigg) , \sin(\theta) \cos\left( \frac{1}{S} y\right) \bigg). $$ 
It follows that  
$$  \frac{\kappa_\alpha}{1 - \tau_\alpha} =  \tilde{\kappa}_\alpha  =  \frac{P}{Q} , \ \ \ \ \  \frac{\kappa_\beta}{1 + \tau_\beta } =  \tilde{\kappa}_\beta  =  \frac{R}{S}  . $$
Now, since $\frac{d}{ds} N_\alpha = - \kappa_\alpha T_\alpha + ( \tau_\alpha - 1 ) B_\alpha$, then 

$$ T_\alpha = \left( P \cos\bigg(\frac{1}{Q} s \bigg), P \sin\bigg(\frac{1}{Q} s \bigg), - Q \right) . $$
Symmetrically, we have
$$ T_\beta = \left( R \cos\left(\frac{1}{S} t  \right), R \cos \theta \sin\left(\frac{1}{S} t \right) + S \sin \theta , R \sin \theta \sin\left(\frac{1}{Q} t \right) - S \cos \theta \right). $$
 
Since $\kappa_\alpha$ and $\kappa_\beta$,
are constant, we differentiate \eqref{eqmincurvatsimpder} with respect to $t$ to obtain

    \begin{multline*}
        \kappa_\alpha \kappa_{\beta} \left[(\tau_\beta + 1 ) -  (\tau_\alpha - 1 ) \right]  \langle  N_\alpha,  N_\beta \rangle  =  \  2 \kappa_\alpha \kappa_\beta \langle N_\alpha, N_\beta \rangle [ 3 \langle T_\alpha, T_\beta \rangle^2 - 1 ] + \\ +  12 \kappa_\alpha \kappa_\beta \langle N_\alpha, T_\beta \rangle \langle T_\alpha, T_\beta \rangle  \langle T_\alpha, N_\beta \rangle .
    \end{multline*}
Now, as $\kappa_\alpha \ne 0$  and  $\kappa_\beta \ne 0$, we have 
\begin{align}
       [(\tau_\beta + 1 )  - (\tau_\alpha - 1 )] \langle  N_\alpha,  N_\beta \rangle  = & \  2  \langle N_\alpha, N_\beta \rangle [ 3 \langle T_\alpha, T_\beta \rangle^2 - 1 ] +   12  \langle N_\alpha, T_\beta \rangle \langle T_\alpha, T_\beta \rangle  \langle T_\alpha, N_\beta \rangle \label{eq_last_theo_2} .
\end{align}
Evaluating $T_\alpha$, $N_\alpha$, $B_\alpha$, $T_\beta$, $N_\beta$ and $B_\beta$ in $(x/Q ,  y/S) = (0 , \pi/2)$ gives
$$ \begin{array}{rclcrcl}
    T_\alpha & = & (P,0,-Q) ,  &  & T_\beta & = & (0,R \cos(\theta) + S \sin(\theta), R \sin(\theta) - S \cos(\theta)) , \\
    N_\alpha & = & (0,1,0) , &  & N_\beta & = & (-1,0,0)  , \\
    B_\alpha & = & (Q,0,P) , &  & B_\beta & = & (0,S \cos(\theta) - R \sin(\theta), S \sin(\theta) + R \cos(\theta)) .
\end{array}$$
Hence 
$$ \begin{array}{rclcrcl}
     \langle T_\alpha, T_\beta  \rangle & = & Q (S \cos\theta- R \sin\theta) , &  &  \langle N_\alpha, N_\beta  \rangle & = & 0 ,  \\ 
     \langle N_\alpha, T_\beta  \rangle & = & R \cos\theta + S \sin\theta , &  &  \langle T_\alpha, N_\beta  \rangle & = & - P .
\end{array}$$
Thus, equation \eqref{eq_last_theo_2} becomes 
\begin{equation}\label{eq_last_theo_3}
    12 PQ (S \cos\theta - R \sin\theta) (R \cos\theta +  S \sin\theta) = 0 .
\end{equation}
If $S \cos\theta = R \sin\theta $ then $\cos\theta = (R/S) \sin\theta$, with $\theta \ne \pi$.  Hence 
$$ R \cos\theta +  S \sin\theta = \left( \frac{R^2}{S} + S \right)\sin\theta  = \frac{1}{S} \sin\theta \ne 0 . $$
Thus, $(S \cos\theta - R \sin\theta) = 0 $ and $ (R \cos\theta +  S \sin\theta) \ne 0$. We have 

$$ T_\beta = \left(0,\frac{1}{S} \sin\theta , 0\right), \ \ \  N_\beta = (-1,0,0), \ \ \  B_\beta = \left(0,0 , \frac{1}{S} \sin\theta \right) . $$
It follows that $\langle T_\alpha, T_\beta  \rangle = \langle B_\alpha, T_\beta  \rangle = 0 $ and $\langle T_\alpha, B_\beta  \rangle = - (Q / S) \sin \theta$. Evaluating in equation \eqref{eqmincurvatsimp} gives $$ \kappa_\alpha \frac{Q}{S}  \sin \theta = 0, $$ 
a contradiction. 

Suppose now that $R \cos\theta = -  S \sin\theta $, then $\cos\theta = - (S/R) \sin\theta$, with $\theta \ne \pi$.  Hence 
$$ S \cos\theta - R \sin\theta = \left(- \frac{S^2}{R} - R \right)\sin\theta  = - \frac{1}{R} \sin\theta \ne 0 . $$
Thus, $(S \cos\theta - R \sin\theta) \ne 0 $ and $ (R \cos\theta +  S \sin\theta) = 0$. We have 

$$ T_\beta = \left(0,0, \frac{1}{R} \cos(\theta)\right), \ \ \  N_\beta = (-1,0,0), \ \ \ B_\beta = \left(0, - \frac{1}{R} \sin(\theta), 0 \right) . $$
It follows that $\langle N_\alpha, T_\beta  \rangle = 0 $, $\langle N_\alpha, B_\beta  \rangle = - \frac{1}{R} \sin\theta $ and $\langle T_\alpha, T_\beta  \rangle = - (Q/R) \sin \theta  $. Evaluating in equation \eqref{eqmincurvatsimpder} gives 
$$\frac{1}{R}\kappa_\alpha \kappa_\beta \sin \theta  = 0, $$ 
a contradiction. 

Thus, we must have $(S \cos\theta - R \sin\theta) = (R \cos\theta +  S \sin\theta) = 0$, a contradiction  with equation \eqref{eq_last_theo_3}.  

\end{proof}

\section{Declarations}

\subsection{Funding}
T. A. Ferreira was supported by CNPq - Conselho Nacional de Desenvolvimento Cient\'ifico e Tecnol\'ogico grant number 162148/2021-6 and CAPES Print Process N° 88887.890396/2023-00. J. P. dos Santos was supported by FAPDF - Fundação de Apoio a Pesquisa do Distrito Federal, grant number 00193-00001678/2024-39, and CNPq, grant number 315614/2021-8.

\subsection{Competing interests}
The authors declare that they have no competing interests relevant to the content of this article.

\subsection{Data Availability}
Data sharing is not applicable to this article, as no datasets were generated or analyzed during the study.

\bibliographystyle{plain}
\bibliography{references}

\begin{thebibliography}{10}

\bibitem{AledoGalvezMira}
Juan~A. Aledo, Jos\'e{}~A. G\'alvez, and Pablo Mira.
\newblock A {D}'{A}lembert formula for flat surfaces in the 3-sphere.
\newblock {\em J. Geom. Anal.}, 19(2):211--232, 2009.

\bibitem{Barros}
Manuel Barros.
\newblock General helices and a theorem of lancret.
\newblock {\em Proceedings of the American Mathematical Society},
  125(5):1503--1509, 1997.

\bibitem{Darboux}
Gaston Darboux.
\newblock {\em Le\c{c}ons sur la th\'{e}orie g\'{e}n\'{e}rale des surfaces et
  les applications g\'{e}om\'{e}triques du calcul infinit\'{e}simal.
  {D}euxi\`eme partie}.
\newblock Chelsea Publishing Co., Bronx, NY, 1972.
\newblock Les congruences et les \'{e}quations lin\'{e}aires aux
  d\'{e}riv\'{e}es partielles. Les lignes trac\'{e}es sur les surfaces,
  R\'{e}impression de la deuxi\`eme \'{e}dition de 1915.

\bibitem{Galvez}
Jos\'{e}~A. G\'{a}lvez.
\newblock Surfaces of constant curvature in 3-dimensional space forms.
\newblock {\em Mat. Contemp.}, 37:1--42, 2009.

\bibitem{GalvezMira}
Jos\'e{}~A. G\'alvez and Pablo Mira.
\newblock Isometric immersions of {$\mathbb{R}^2$} into {$\mathbb{R}^4$} and
  perturbation of {H}opf tori.
\newblock {\em Math. Z.}, 266(1):207--227, 2010.

\bibitem{LopezHasanis}
Thomas Hasanis and Rafael L\'{o}pez.
\newblock Classification and construction of minimal translation surfaces in
  {E}uclidean space.
\newblock {\em Results Math.}, 75(1):Paper No. 2, 22, 2020.

\bibitem{Hoffman}
David~A. Hoffman.
\newblock Surfaces of constant mean curvature in manifolds of constant
  curvature.
\newblock {\em J. Differential Geometry}, 8:161--176, 1973.

\bibitem{Munteanu}
Jun-ichi Inoguchi, Rafael L\'{o}pez, and Marian-Ioan Munteanu.
\newblock Minimal translation surfaces in the {H}eisenberg group {${\rm
  Nil}_3$}.
\newblock {\em Geom. Dedicata}, 161:221--231, 2012.

\bibitem{Kitagawa}
Yoshihisa Kitagawa.
\newblock Periodicity of the asymptotic curves on flat tori in {$S^3$}.
\newblock {\em J. Math. Soc. Japan}, 40(3):457--476, 1988.

\bibitem{Lopez}
Rafael L\'{o}pez.
\newblock Minimal translation surfaces in hyperbolic space.
\newblock {\em Beitr. Algebra Geom.}, 52(1):105--112, 2011.

\bibitem{LopezMunteanu}
Rafael L\'{o}pez and Marian~Ioan Munteanu.
\newblock Minimal translation surfaces in {${\rm Sol}_3$}.
\newblock {\em J. Math. Soc. Japan}, 64(3):985--1003, 2012.

\bibitem{LopezPerdomo}
Rafael L\'{o}pez and \'{O}scar Perdomo.
\newblock Minimal translation surfaces in {E}uclidean space.
\newblock {\em J. Geom. Anal.}, 27(4):2926--2937, 2017.

\bibitem{ManfioSantos}
Fernando Manfio and Jo\~ao~Paulo dos Santos.
\newblock Helicoidal flat surfaces in the 3-sphere.
\newblock {\em Math. Nachr.}, 292(1):127--136, 2019.

\bibitem{MoruzMunteanu}
Marilena Moruz and Marian~Ioan Munteanu.
\newblock Minimal translation hypersurfaces in {$\mathbb{E}^4$}.
\newblock {\em J. Math. Anal. Appl.}, 439(2):798--812, 2016.

\bibitem{Spivack4}
Michael Spivak.
\newblock {\em A comprehensive introduction to differential geometry. {V}ol.
  {IV}}.
\newblock Publish or Perish, Inc., Boston, MA, 1975.

\bibitem{XuHan}
Zhonghua~Hou. Xu~Han.
\newblock Translation surfaces in lie groups with constant gaussian curvature.
\newblock {\em Transformation Groups}, 2024.

\bibitem{Yoon}
D.~W. Yoon.
\newblock Minimal translation surfaces in {$\mathbb{H}^2\times\mathbb{R}$}.
\newblock {\em Taiwanese J. Math.}, 17(5):1545--1556, 2013.

\end{thebibliography}

\end{document}